\documentclass[a4paper,11pt,times,authoryears,reqno]{amsart}

\usepackage{geometry}
\usepackage{bm}
\usepackage{amssymb,amsfonts}
\usepackage{graphicx}
\usepackage{mathrsfs}
\usepackage{subcaption}
\usepackage{epstopdf}
\usepackage{float}
\usepackage{enumerate}
\usepackage{booktabs}
\usepackage{adjustbox}
\usepackage{pdflscape}
\usepackage[colorlinks,citecolor=blue]{hyperref}

\newtheorem{theorem}{Theorem}[section]
\newtheorem{lemma}[theorem]{Lemma}

\theoremstyle{definition}

\theoremstyle{remark}

\newtheorem{proposition}[theorem]{Proposition}

\numberwithin{equation}{section}
\numberwithin{equation}{section}

\numberwithin{equation}{section}

\def\no{\nonumber}

  \def\no{\nonumber}

\def\bW{{\mathbf W}}

  \def\bD{{\mathbf D}} 
  
 \def\bH{{\mathbf H}}

 \def\bW{{\mathbf W}} \def\bX{{\mathbf X}}

\def\bx{{\mathbf x}} \def\by{{\mathbf y}} 

\def\bbeta{{\boldsymbol{\beta}}}

 \def\bdelta{{\boldsymbol{\delta}}}
 \def\bgamma{{\boldsymbol{\gamma}}}

 \def\bGamma{{\boldsymbol{\Gamma}}}

\def\bPhi{{\boldsymbol{\Phi}}}

  \def\bxi{{\boldsymbol{\xi}}}
       
\def\bSigma{{\boldsymbol{\Sigma}}}

\def\bvt{{\boldsymbol \vartheta}} 
\def\hbbeta{\widehat{\boldsymbol \beta}}

\def\real{\mathop{{\rm I}\kern-.2em\hbox{\rm R}}\nolimits}

\def\1overn{\frac{1}{n}}

\def\bel{\begin{eqnarray}\label}  \def\eel{\end{eqnarray}}
\def\bes{\begin{eqnarray*}}  \def\ees{\end{eqnarray*}}

\def\PT{\hbbeta_{1,\tau}^{\rm PT}}
\def\FM{\hbbeta_{1,\tau}^{\rm FM}}
\def\SM{\hbbeta_{1,\tau}^{\rm SM}}
\def\PS{\hbbeta_{1,\tau}^{\rm PS}}
\def\SS{\hbbeta_{1,\tau}^{\rm S}}
\def\btau{\bbeta_{1,\tau}}
\def\Ups{{\bf\Upsilon}}

\begin{document}

\title{Improved Quantile Regression Estimators when the Errors are Independently and Non-identically Distributed}


\author{Bahad{\i}r Y\"{u}zba\c{s}{\i}$^\dag$, Yasin Asar$^\S$, Ahmet Demiralp$^\clubsuit$ and M.\c{S}amil \c{S}{\i}k$^\ddag$}

\date{\today}
\maketitle

{\footnotesize
\center { \text{  $^\dag\ddag\clubsuit$ Department of Econometrics}\par
  { \text{ Inonu University}}\par
  {\text{ Malatya 44280, Turkey}}\par
  { \texttt{E-mail address: $^\dag$b.yzb@hotmail.com, $^\ddag$mhmd.sml85@gmail.com and $^\clubsuit$ahmt.dmrlp@gmail.com}}\par

  \vskip 0.2 cm

  \text{  $^\S$ Department of Mathematics-computer Sciences
}\par
  { \text{ Necmettin Erbakan University}}\par
  {\text{ Konya 42090, Turkey}}\par
  { \texttt{E-mail address: yasar@konya.edu.tr, yasinasar@hotmail.com}}

}}


\renewcommand\leftmark {\centerline{  \rm Quantile Shrinkage Estimation when the Errors are i.ni.d. Distributed}}
\renewcommand\rightmark {\centerline{ \rm  Quantile Shrinkage Estimation when the Errors are i.ni.d. Distributed}}

\renewcommand{\thefootnote}{}

\footnote {\par

Corresponding author : Bahad{\i}r Y\"{u}zba\c{s}{\i} 
}

\begin{abstract}
In a classical regression model, it is usually assumed that the explanatory variables are independent of each other and error terms are normally distributed. But when these assumptions are not met, situations like the error terms are not independent or they are not identically distributed or both of these, LSE will not be robust. Hence, quantile regression has been used to complement this deficiency of classical regression analysis and to improve the least square estimation (LSE). In this study, we consider preliminary test and shrinkage estimation strategies for quantile regression models with independently and non-identically distributed (i.ni.d.) errors. A Monte Carlo simulation study is conducted to assess the relative performance of the estimators. Also, we numerically compare their performance with Ridge, Lasso, Elastic Net penalty estimation strategies. A real data example is presented to illustrate the usefulness of the suggested methods. Finally, we obtain the asymptotic results of suggested estimators
\end{abstract}

\maketitle

\section{Introduction}\label{sec1}

Consider a linear regression model
\begin{equation}
y_{i}=\bx_{i}'\bm{\beta }+\varepsilon _{i},\ \ \
i=1,2,...,n,  \label{lin.mod}
\end{equation}%
where $y_{i}$'s are responses, $\bx_{i}=\left(
x_{i1},x_{i2},...,x_{ip}\right)'$ are observation points, $%
\bm{\beta }=\left( \beta ,\beta _{2},...,\beta _{p}\right)
'$ is a vector of unknown regression coefficients, $\varepsilon _{i}{}$'$s$
are unobservable random errors and the superscript $\left( '\right) $
denotes the transpose of a vector or matrix. 
In this study, we consider that the design matrix has rank $p$ ($p\leq n$).

In a linear regression model, it is usually assumed that the
explanatory variables are independent of each other and error terms are normally distributed. However, in many areas, including econometrics, survival analysis and ecology, etc. data doesn't satisfy these assumptions. Firstly introduced by \cite{koenker1978regression}, quantile regression has been used to complement this deficiency of classical regression analysis and to improve the least square estimation (LSE). 

When using LSE, one usually obtains representation of the relationship between explanatory variables and dependent variables only at one point. However, it doesn't give information about the relationship at any other possible points of interest. Quantile regression gives a complete representation of the variables in the model and does not make any distributive assumption about the error term in the model. The main advantage of quantile regression against LSE is its flexibility in modeling the data in heterogeneous conditional distributions. 

On the basis of the quantile regression lies the expansion of the regression model to the conditional quantities of the dependent variable. A quantile is a one of the equally segmented subsets of a sorted sample of a population. If we need to formulate a quantile mathematically, let $\mathcal{F}$ be the distribution function of a random variable $Y=\left(y,y_{2},...,y_{n}\right)$ such that $\label{eq:F_Y}
\mathcal{F}_{Y}\left(y\right)=P(Y\le y)=\tau$ and $0\leq\tau\leq1$, the $\tau^{th}$ quantile function of $Y$, $\mathcal{Q}_\tau(\by)$ is defined to be 
\begin{equation}
\mathcal{Q}_\tau(\by)=\mathcal{F}_{Y}^{-1}(\tau)=\inf\left\{\by\vert \mathcal{F}_{Y}\left(\by\right)\ge\tau\right\}.
\end{equation}
For a random sample $\by=\left(y_{1},y_{2},\dots,y_{n}\right)$ with empirical distribution function $\widehat{\mathcal{F}}_{Y}\left(\tau\right)$ by solving minimization of error squares also an estimation of the $\tau^{th}$  quantile regression coefficients ($\FM$) can be defined by solving a minimization of absolute errors problem. \cite{wu2009variable} considered quantile regression for variable selection based on SCAD and adaptive-LASSO penalties under independently and identically distributed (i.i.d.) and independently and non-identically (i.ni.d.) error assumptions. They proved the oracle properties of SCAD and adaptive-LASSO penalized quantile regression. \cite{wei2012multiple} proposed a multiple imputation estimator for parameter estimation in a quantile regression model when some covariates are missing at random. \cite{yoon2013penalized} studied a variable selection problem in penalized regression models with autoregressive error terms. They proposed a computational algorithm that enables us to select a relevant set of variables and also the order of autoregressive error terms simultaneously and compared performances of adaptive LASSO (Least Absolute Shrinkage and Selection Operator), bridge, and SCAD (Smoothly Clipped Absolute Deviation) estimators. \cite{ahmed2014penalty} provided a collection of topics outlining pretest and Stein-type shrinkage estimation techniques in a variety of regression modeling problems. \cite{li2015quantile} applied quantile correlation (QCOR) and quantile partial correlation (QPCOR) measures to quantile autoregressive (QAR) models to extend the classical Box-Jenkins approach to quantiles autoregressive models. They introduced the quantile autocorrelation function (QACF) and the quantile partial autocorrelation function (QPACF). Moreover, they showed the usefulness of the proposed methods on the large sample results of the QAR estimates and the quantile version of the Ljung-Box test. \cite{schumacheracensored} proposed a stochastic approximation of the EM (SAEM) algorithm which permits easy and fast estimation of the parameters of autoregressive models when censoring is present and as a byproduct, enables predictions of unobservable values of the response variable. They also provide an implementation via the R package \textit{ARCensReg}. R and the package \textit{ARCensReg} are open-source software projects and can be freely downloaded from CRAN: \url{https://cran.r-project.org/web/packages/ARCensReg/index.html}. The books by \cite{koenker2005quantile} and \cite{davino2013quantile} are an excellent source for various properties of Quantile Regression as well as many computer algorithms. In this case, some biased estimations, such as shrinkage estimation, principal components estimation (PCE), ridge estimation \cite{hoerl1970ridge} 
were proposed to improve the
least square estimation (LSE). To combat multicollinearity, \cite{yuzbasi2016shrinkage,yuzbasi2017improved} proposed the pretest and Stein-type ridge regression estimators for linear and partially linear models. \cite{yuzbasi2017pretest} considered preliminary test and shrinkage estimation strategies for quantile regression models. \cite{yuzbasi2017improving} applied a quantile regression approach is used to model the respiratory mortality using the mentioned explanatory variables. Moreover, improved estimation techniques such as preliminary testing and shrinkage strategies are also obtained when the errors are autoregressive. 

The novelty of this paper apart from the above studies is considered preliminary test and shrinkage estimation strategies for quantile regression models when the errors are independently and non-identically distributed, and the organization of this paper as follows: the full and sub-model estimators are given in Section~\ref{sec2}. Moreover, the preliminary test quantile estimator, shrinkage quantile estimators and the positive part of the shrinkage estimator are proposed together with a brief definition of the penalized estimations in this section. The asymptotic properties of the pretest and shrinkage estimators estimators are obtained in Section~\ref{sec3}. The design and the results of a Monte Carlo simulation study including a comparison with other penalty estimators are given in Section~\ref{sec4}. A real data example is given for illustrative purposes in Section~\ref{sec5}.  The concluding remarks are presented in Section~\ref{sec6}.

\section{Estimation Strategies}\label{sec2}
Linear regression model given in $\eqref{lin.mod}$ would be written in a partitioned form as follows
\begin{equation}
y_i=\bx_{1i}'\bbeta_1+\bx_{2i}'\bbeta_2+\varepsilon _{i},\ \ \
i=1,2,\dots,n,  \label{part.lin}
\end{equation}
where $p=p_1+p_2$, $\bbeta_1$ and $\bbeta_2$ are parameters of $p_1$ and $p_2$ respectively. $\bx_i=\left(\bx_{1i}',\bx_{2i}'\right)$ and $\varepsilon _{i}$ are errors with the same joint distribution function $\mathcal{F}$.
The conditional quantile function of response variable $y_i$  would be written as follows
\begin{equation}\label{eq:x_tau}
\mathcal{Q}_{y}(\tau\vert \bx_i)=\bx_{1i}'\bbeta_{1,\tau}+\bx_{2i}'\bbeta_{2,\tau},\ \ \ 0<\tau< 1 
\end{equation}
The main interest here is to test the hypothesis $\bH_0:\bbeta_{2,\tau}=0$. Full model quantile regression estimator is the value that minimizes the following problem $\hbbeta_\tau^{\rm FM}=\underset{\bbeta\in\Re^{p}}{\min} \sum_{i=1}^{n}\rho_\tau (y_i-\bx'_{i}\bbeta)$. Sub model quantile regression estimator is $\hbbeta_\tau^{\rm SM}=\left(\hbbeta_{1,\tau}^{SM},\mathbf{0}\right)$. Also $\hbbeta_{1,\tau}^{\rm SM}=\underset{\bbeta_1\in\Re^{p_1}}{\min} \sum_{i=1}^{n}\rho_\tau (y_i-\bx'_{1i}\bbeta_1)$.

Let $Y,Y_{2},...$ be independent random variables with distribution functions $F,F_{2},...$ and suppose that the $\tau^{th}$ conditional quantile function $\mathcal{Q}_\tau(\by) = \bx' \bbeta_\tau$
is linear in the covariate vector $\bx$. The conditional distribution functions of the $Y_{i}$'s will be written as  $\ P(Y_{i}\ < y\vert\bx_{i})=\mathcal{F}_{{Y}_{i}}(y\vert\bx_i)=\mathcal{F}_i(y)$, and so $\mathcal{Q}_\tau(\by_i)=\mathcal{F}_{{Y}_{i}}^{-1}(\tau\vert\bx_i)\equiv\xi(\tau)$. The distribution function ${\mathcal{F}_i}$ are absolutely continuous, with continuous densities $f_i(\xi)$ uniformly bounded away from $0$ and $\infty$ at the points $\xi_i(\tau)$, $i=1,2,\dots .$ We will use the following continuity conditions to discover the asymptotic behavior of the estimators:
\begin{enumerate}
\item[(i)] $\lim_{n\rightarrow\infty} \frac{1}{n}\sum_{i=1}^n\ \bx_i\bx_i'=\bD, \bD_{0}=\frac{1}{n}\bX'\bX$
\item[(ii)] $\lim_{n\rightarrow\infty} \frac{1}{n}\sum_{i=1}^n\ f_i\left(\bxi_i(\tau)\right)\bx_i\bx_i'=\bD_1$, where $\bD_0$ and $\bD_1$ are positive definite matrices.
\end{enumerate}

\begin{theorem}\label{teo_dist_FM_inid} Quantile regression model with i.ni.d errors under assumptions (i) and (ii)
\begin{equation}
\sqrt{n}\left(\hbbeta_{\tau}^{\rm FM} -\bbeta_\tau\right) \rightarrow^{\hspace{-0.3cm}D} \mathcal{N} \left(\boldsymbol{0},\tau(1-\tau) \bGamma^{-1} \right),
\text{ where } \bGamma^{-1} = \bD_1^{-1}\bD_{0}\bD_1^{-1}.
\end{equation}
\end{theorem}
\begin{proof}
The proof is given by \cite{koenker2005quantile}.
\end{proof}

$\bGamma$ is partitioned in blocks as $\boldsymbol{\Gamma}=
\left(
\begin{array}{cc}
\bGamma_{11}&\bGamma_{12}\\
\bGamma_{21}&\bGamma_{22}\\
\end{array}
\right)$  and 
the test statistic for $H_0: \bbeta_{2,\tau}=\boldsymbol{0}$ is given by
\begin{equation} \label{test statistic}
\mathcal{W}_n = \frac{n}{\tau(1-\tau)} \left(\hbbeta_{2,\tau}^{\rm FM}\right)'\bGamma_{22.1} \hbbeta_{2,\tau}^{\rm FM}
\end{equation}
$\bGamma_{22.1} = {\bGamma}_{22} - {\bGamma}_{21}{\bGamma}_{11}^{-1}{\bGamma}_{12}$.
Under the null hypothesis $H_0$, $\mathcal{W}_{n}$ has the chi-square distribution with $p_2$ degrees of freedom (d.f.).
Hence, we are ready to define pretest and shrinkage estimations as follows. The preliminary test (PT) estimator of $\btau$ is defined by
\begin{equation}
\label{beta_PT}
\PT=\FM-\left(\FM-\SM \right ) \textrm{I}\left(\mathcal{W}_{n}\leq  \chi^2_{p_2}\right),
\end{equation}
where $\textrm{I}\left(A\right)$ is the indicator function of the set $A$. The shrinkage quantile regression estimator $\SS$ of $\btau$ is proposed as
\begin{equation}
\label{beta_S}
\SS =\SM +\left(\FM -\SM \right ) (1-(p_2-{2})\mathcal{W}_{n}^{-1}), p_{2}\ge 3.
\end{equation}
The positive part of the shrinkage estimator $\PS$ of $\btau$ is also proposed as
\begin{equation}
\label{beta_PS}
\PS = \SM +\left(\FM -\SM \right ) (1-(p_2-{2})\mathcal{W}_{n}^{-1})^+.
\end{equation}

\subsection{Penalized Estimation}
In \cite{yi2017semismooth}, the penalized estimators for quantile are given by
\begin{equation}
\bm{\widehat{\beta}}=\underset{\bm{\bm{\beta}}}{\arg \min}\sum_{i}\rho(y_i-\bx_{i}'\bbeta)+\lambda\ P(\bbeta),
\end{equation}
where $\rho$ is a quantile loss function, $P$ is a penalty function and $\lambda$ is a tuning parameter.  
\begin{equation*}
P\left(\bbeta\right)\equiv P_{\alpha}(\bbeta)=\alpha\Vert\bbeta\Vert_1+\frac{(1-\alpha)}{2}\Vert\bbeta\Vert_2^2
\end{equation*}
which is the lasso penalty for $\alpha = 1$ \cite{tibshirani1996regression}, the ridge penalty for $\alpha=0$ \cite{hoerl1970ridge} and the elastic-net penalty for $0\le \alpha \le {1}$ \cite{zou2005regularization}. 

\section{Asymptotics}\label{sec3}
Consider a sequence of local alternatives $\left\{K_n\right\}$ given by
\begin{equation*}
\ K_n:\bbeta_{2,\tau}=\frac{\bgamma}{\sqrt{n}}
\end{equation*}
where $\bgamma=\left(\gamma_{1},\gamma_{2},\dots,\gamma_{p_{2}}\right)'\in \Re^{p_{2}}$ is a fixed vector. If $\bgamma=\boldsymbol{0}_{p_2}$, then the null hypothesis is true. Moreover, we consider the following proposition to establish the asymptotic properties of the estimators.
\begin{proposition}\label{prop_vector_dist} 
Let $\bvt _{1} = \sqrt{n}\left(\hbbeta_{1, \tau}^{\rm FM}-%
\btau\right)$, $\bvt _{2} =\sqrt{n}\left( \SM-%
\btau \right)$ and $\bvt _{3} =\sqrt{n}\left(\FM-\SM\right)$. Under the regularity assumptions A1 and A2, Theorem~\ref{teo_dist_FM_inid} and the local alternatives $\left\{ K_{n}\right\}$, as $n\rightarrow \infty$ we have the joint distributions are given as follows:
$$\left(
\begin{array}{c}
\bvt _{1} \\
\bvt _{3}%
\end{array}%
\right) \sim\mathcal{N}\left[ \left(
\begin{array}{c}
\boldsymbol{0 }_{p_1} \\
-\bdelta%
\end{array}%
\right) ,\left(
\begin{array}{cc}
\tau(1-\tau)\bGamma_{11.2}^{-1} & \bSigma_{12} \\
\bSigma_{21} & \bPhi%
\end{array}%
\right) \right]$$

$$\left(
\begin{array}{c}
\bvt _{3} \\
\bvt _{2}%
\end{array}%
\right) \sim\mathcal{N}\left[ \left(
\begin{array}{c}
-\bdelta \\
\bdelta%
\end{array}%
\right) ,\left(
\begin{array}{cc}
\bPhi & \bSigma^* \\
\bSigma^* & \tau(1-\tau)\bGamma_{11}^{-1}%
\end{array}%
\right) \right]$$ \\
where $\bdelta= \bGamma_{11}^{-1}\bGamma_{12}\bgamma$,  $\bPhi=\tau(1-\tau)\bGamma_{11}^{-1}\bGamma_{12}\bGamma_{22.1}^{-1}\bGamma_{21}\bGamma_{11}^{-1}$, $\bSigma_{12}=-\tau(1-\tau)\bGamma_{12}\bGamma_{21}\bGamma_{11}^{-1}$, $\bSigma^*=\tau(1-\tau)\left(\bGamma_{11.2}^{-1}+\bGamma_{12}\bGamma_{21}\bGamma_{11}^{-1} - \bGamma_{11}\right)$ and $\bGamma_{11.2} = \bGamma_{11} - \bGamma_{12}\bGamma_{22}^{-1}\bGamma_{21}$.
\end{proposition}
\begin{proof}
See Appendix.
\end{proof}
\subsection{The performance of Bias} 
The asymptotic bias of an estimator $\hbbeta%
_{1}^{\ast }$ is defined as
\begin{eqnarray*}
\mathcal{B}\left( \hbbeta_{1,\tau}^{\ast }\right) =\mathbb{E} \underset{%
n\rightarrow \infty }{\lim }\left\{\sqrt{n}\left( \hbbeta_{1,\tau}^{\ast }-
\bbeta_{1,\tau}\right) \right\}.
\end{eqnarray*}
Hence we can give the following theorem.
\begin{theorem}
\label{ADB-Low-Linear}
Under the assumed regularity conditions in (i) and (ii), the Proposition \ref{prop_vector_dist},  the Theorem~\ref{teo_dist_FM_inid}  and $\left\{K_n\right\}$, the expressions for asymptotic biases for listed estimators are:
\begin{eqnarray*}
\mathcal{B}\left(\FM \right) &=&
\boldsymbol{0} \cr \\
\mathcal{B}\left( \SM \right) &=&
\bdelta\\
\mathcal{B}\left( \PT\right) &=&\bdelta H_{p_{2}+2}\left( \chi _{p_{2},\alpha
}^{2};\Delta \right) ,\text{ } \cr\\
\mathcal{B}\left( \SS\right) &=&(p_2-2)\bdelta \mathbb{E}\left\{\chi _{p_{2}+2}^{-2}\left(\Delta\right)\right\}
\text{ } \cr\\
\mathcal{B}\left( \PS \right) &=&\bdelta H_{p_{2}+2}\left( p_2-2;\Delta \right)+(p_{2}-2)\bdelta \mathbb{E}\left\{ \chi _{p_{2}+2 }^{-2}\left(
\Delta \right) \textrm{I}\left( \chi _{p_{2}+2 }^{2}\left( \Delta \right)
>p_{2}-2\right) \right\} 
\end{eqnarray*} 
where ${H}_{v}\left( x,\Delta \right)$ is the cumulative distribution
function of the non-central chi-squared distribution with non-centrality
parameter $\Delta =\frac{\bdelta'\bGamma_{22.1}\bdelta}{\tau(1-\tau)}$ and $v$ degree of freedom, and
\begin{equation*}
\mathbb{E}\left( \chi _{v}^{-2j}\left( \Delta \right) \right)
=\int\nolimits_{0}^{\infty }x^{-2j}d\mathbb{H}_{v}\left( x,\Delta \right) .
\end{equation*}
\end{theorem}
\begin{proof}
See Appendix.
\end{proof}
Now, we define the following asymptotic quadratic bias $\left(\mathcal{QB}\right)$ of an estimator $\hbbeta_{1,\tau}^*$
by converting them into the quadratic form since the bias expression of all the estimators are not in the scalar form.
\begin{equation}\label{eq:asqubi}
\mathcal{QB}\left(\hbbeta_{1,\tau}^*\right)=\mathcal{B}\left(\hbbeta_{1,\tau}^*\right)' \bGamma_{11.2}\mathcal{B}\left(\hbbeta_{1,\tau}^*\right).
\end{equation} 
Using the definition in $\eqref{eq:asqubi}$, the asymptotic distributional quadratic bias of the
estimators are presented below.
\begin{eqnarray*}
\mathcal{QB}\left(\FM \right) &=&
\boldsymbol{0}\cr \\
\mathcal{QB}\left( \SM \right) &=&
\bdelta'\bGamma_{11.2}\bdelta \cr \\
\mathcal{QB}\left( \PT \right) &=& \bdelta'\bGamma_{11.2}\bdelta \left[H_{p_{2}+2}\left( \chi _{p_{2},\alpha
}^{2};\Delta \right)\right]^2 ,\text{ } \cr \\
\mathcal{QB}\left( \SS\right) &=& (p_{2}-2)^2\bdelta'\bGamma_{11.2}\bdelta \left[\mathbb{E}\left( \chi _{p_{2}+2
}^{-2}\left( \Delta\right)\right)\right]^2, \cr \\
\mathcal{QB}\left( \PS \right) &=&\bdelta'\bGamma_{11.2}\bdelta \Big[ \mathbb{H}_{p_{2}+2}\left(p_{2}-2 ;\Delta\right) +(p_{2}-2)\mathbb{E}\left\{ \chi _{p_{2}+2 }^{-2}\left(\Delta \right) \textrm{I}\left( \chi _{p_{2}+2 }^{2}\left( \Delta \right)>p_{2}-2\right) \right\}\Big]^2.
\end{eqnarray*}%

The quadratic bias of the $\FM$, i.e, $\FM$ is a unbiased estimator. On the other hand, if the null hypothesis is true, all others are unbiased. When $\Delta>0$, the quadratic bias of $\SM$ is unbounded function of $\Delta$ while all the remaining estimators are bounded. Since the quadratic bias of $\PT$ is a function of $\Delta$, it starts from zero, increases to a point, then decreases gradually to zero. The quadratic bias functions of $\S$ and $\PS$ starts from zero when the null hypothesis is true, and increases to a point and then decreases toward to zero because $E(\chi^{-2}_{p_2+2}\left( \Delta \right))$ is decreasing convex function of $\Delta$. Note that the graph of the quadratic bias of $\PS$ remain below the graph of the quadratic bias of $\SS$.
\subsection{The performance of Risk}
The asymptotic distributional risk of an estimator $\hbbeta_{1, \tau}^{\ast }$ is defined as
\begin{equation}
\mathcal{R}\left( \hbbeta_{1,\tau}^{\ast} \right) = {\rm tr}\left(  \Ups\left(\hbbeta_{1,\tau}^{\ast}\right)\boldsymbol{W} \right) 
\label{risk_def}
\end{equation}
where $\boldsymbol{W}$ is a positive definite matrix of weights with dimensions of $p_1 \times p_1$, and $\Ups\left(\hbbeta_{1,\tau}^{\ast}\right)$ is the asymptotic covariance matrix of an estimator $\hbbeta_{1, \tau}^{\ast }$ is defined as
\begin{equation*}
\Ups \left( \hbbeta_{1,\tau}^{\ast} \right) = \mathbb{E}\left\{\underset{n\rightarrow \infty }{\lim }{n}\left( \hbbeta_{1,\tau}^{\ast} -\bbeta_{1,\tau}\right)\left( \hbbeta_{1,\tau}^{\ast}-\bbeta_{1,\tau}\right) ^{'}\right\}.
\end{equation*}
Based on the computations regarding the asymptotic covariances, we present the risks of the estimators $\FM$, $\SM$, $\PT$, $\SS$ and $\PS$ respectively in the following theorem.
\begin{theorem} Under the local alternatives $\left\{ K_{n}\right\} $ and assuming the regularity conditions (i) and (ii), the risks of the estimators are:
\label{risk}
\begin{eqnarray*}
\mathcal{R}\left( \FM\right)&=&
\tau(1-\tau){\rm tr}\left(\bW \bGamma_{11.2}^{-1}\right)\\
\mathcal{R}\left( \SM \right)&=& \tau(1-\tau) {\rm tr}\left(\bW \bGamma_{11}^{-1}\right)+\bdelta'\bW \bdelta \\
\mathcal{R}\left( \PT \right) 
&=& \tau(1-\tau){\rm tr}\left( \bW\bGamma_{11.2}^{-1}\right)-2{\rm tr}\left( \bW\bSigma_{21} \right)\mathbb{H}_{p_{2}+2}\left(
\chi _{p_{2},\alpha }^{2};\Delta \right)\no\\
&&+{\rm tr}\left(\bW \bdelta\bdelta'\bPhi^{-1}\bSigma_{21} \right)\left[\mathbb{H}_{p_{2}+2}\left( \chi_{p_{2}}^{2};\Delta \right)-2 \mathbb{H}_{p_{2}+4}\left( \chi _{p_{2},\alpha
}^{2};\Delta \right)\right]\no\\
&&+{\rm tr}\left( \bW\bPhi \right)\mathbb{H}_{p_{2}+2}\left(
\chi _{p_{2},\alpha }^{2};\Delta \right)+ \bdelta'\bW \bdelta\mathbb{H}_{p_{2}+4}\left( \chi _{p_{2},\alpha
}^{2};\Delta \right) \\    
\mathcal{R}\left( \SS\right) &=&
\tau(1-\tau){\rm tr}\left(\bW\bGamma_{11.2}^{-1} \right) -2(p_{2}-2){\rm tr}\left( \bW \bSigma_{21}\right)\left\{\chi _{p_{2}+2}^{-2}\left(\Delta\right)\right\}\no\\
&&-2(p_{2}-2){\rm tr}\left( \bW \bdelta \bdelta'\bPhi^{-1}\bSigma_{21}\right)\left[ \mathbb{E}\left\{\chi _{p_{2}+4}^{-2}\left(\Delta\right)\right\}+\mathbb{E}\left\{\chi _{p_{2}+2}^{-2}\left(\Delta\right)\right\}\right]\no \\
&&+(p_{2}-2)^2\left({\rm tr}\left( \bW\bPhi \right) \mathbb{E}\left\{\chi _{p_{2}+2}^{-4}\left(\Delta\right) \right\}+\bdelta'\bW\bdelta\mathbb{E}\left\{\chi _{p_{2}+4}^{-4}\left(\Delta\right) \right\}\right)\\
\end{eqnarray*}
\begin{eqnarray*}
\mathcal{R}\left( \PS \right) &=& \mathcal{R}\left( \SS\right)-2{\rm tr}\left(\bW\bSigma_{21}\right)\mathbb{E}\left\{\left(1-(p_{2}-2)\chi _{p_{2}+2}^{-2}\left(\Delta\right)\right)\textrm{I}\left(\chi _{p_{2}+4}^{2}\left(\Delta\right)\leq p_{2}-2\right)\right\}\no\\
&&-2{\rm tr}\left(\bW\bdelta\bdelta^{'}\bPhi^{-1}\bSigma_{21}\right)\mathbb{E}\left\{1-(p_{2}-2)\chi _{p_{2}+4}^{-2}\left(\Delta\right)\textrm{I}\left(\chi _{p_{2}+4}^{2}\left(\Delta\right)\leq p_{2}-2\right)\right\}\no \\
&&-2{\rm tr}\left(\bW\bdelta\bdelta^{'}\bPhi^{-1}\bSigma_{21}\right)\mathbb{E}\left\{1-(p_{2}-2)\chi _{p_{2}+2}^{-2}\left(\Delta\right)\textrm{I}\left(\chi _{p_{2}+2}^{2}\left(\Delta\right)\leq p_{2}-2\right)\right\}\no \\
&&-(p_{2}-2)^2\left[{\rm tr}\left(\bW\bPhi \right) +\bdelta'\bW\bdelta\right]\mathbb{E}\left\{\chi _{p_{2}+2}^{-4}\left(\Delta \right)\textrm{I}\left(\chi _{p_{2}+2}^{2}\left(\Delta\right)\leq p_{2}-2 \right)\right\}\no\\
&&+{\rm tr}\left(\bW\bPhi\right)\mathbb{H}_{p_{2}+2}\left(
p_{2}-2;\Delta \right)+\bdelta'\bW\bdelta\mathbb{H}_{p_{2}+4}\left(
p_{2}-2;\Delta \right)
\end{eqnarray*}
\end{theorem}
\section{Simulations}\label{sec4}

We conduct Monte-Carlo simulation experiments to study the performances of the proposed estimators under various practical settings. In order to generate the response variables, we use
\begin{equation*}
y_i=\bx_i'\bbeta+\varepsilon_i,\ i=1,\dots,n,
\end{equation*}%
where $\bx_i$'s are standard normal. The correlation between the $j$th and $k$th components of $\bx$ equals to $0.5^{|j-k|}$ and also $\varepsilon_i$'s are independently and non-identically distributed. 

\subsection{Asymptotic Investigations}
 We consider the regression coefficients are set
$\bbeta=\left( \bbeta_{1}',\bbeta_{2}'\right)' =\left( \mathbf{1}'_{p_1},\mathbf{0}_{p_2}'\right)'$, where $\mathbf{1}_{p_1}$ and $\mathbf{0}_{p_2}$ mean the vectors of 1 and 0 with dimensions $p_1$ and $p_2$, respectively. In order to investigate the behavior of the estimators, we define $\Delta^{\ast}=\left\Vert \bbeta -\bbeta_{0}\right\Vert $, where $\bbeta_{0}=\left( \mathbf{1}'_{p_1},\mathbf{0}_{p_2}'\right)'$ and $\left\Vert \cdot \right\Vert $ is the Euclidean norm. 
Also, it is taken $n=60$, $p_1=p_2=5$, $\alpha=0.01,0.05,0.10,0.25$ and $\sigma=1$. Furthermore, we consider errors are taken from $0.5\mathcal{N}(0,1) + 0.5\mathcal{N}(0,100)$. The performance of an estimator $\hbbeta_{\tau}^{\ast}$ was evaluated by
using the model error (ME) criterion which is defined by
\begin{equation*}
\textnormal{ME}\left(\hbbeta_{\tau}^{\ast}\right) =\left( \hbbeta_{\tau}^{\ast}-\bbeta\right)'\left( \hbbeta_{\tau}^{\ast}-\bbeta\right)
\end{equation*}

\begin{figure}[ht]
\centering
\includegraphics[height=4cm,width=14cm]{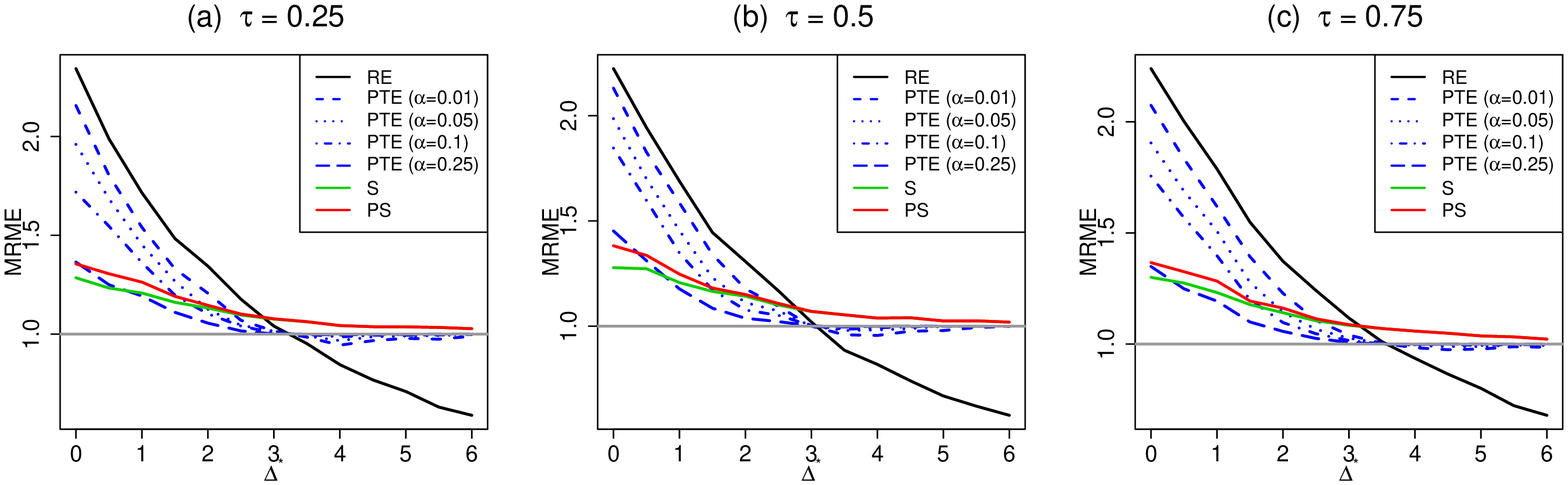}
\caption{MRME of the estimators as a function of the
$\Delta^{\ast}$}
\label{fig:delta}
\end{figure}

In Figure ~\ref{fig:delta}, we plot the median relative model error (MRME) which is defined by ${\rm MRME}\left(\hbbeta_{\tau}^{\ast}\right)=\frac{\hbbeta_{\tau}^{\rm FM}}{\hbbeta_{\tau}^{\ast}}$ versus as a function of $\Delta^{\ast}$. When $\Delta^{\ast}=0$, the performance of $\SM$ outshines all proposed estimators. On the other hand, if $\Delta^{\ast}>0$, then the performance of $\SM$ loses and goes to zero. The performance of $\PT$ is better than $\FM$, $\SS$ and $\PS$ in case of $\Delta^{\ast}=0$. However,  $\PT$ loses its efficiency for intermediate values of $\Delta^{\ast}$, even worse than $\FM$, after that it acts like $\FM$ for larger values of $\Delta^{\ast}$. Clearly, $\PS$ performs better than $\SS$ for each values of $\Delta^{\ast}$. Both shrinkage estimators outperforms $\FM$ regardless the correctness of the selected sub-model at hand.

\subsection{Performance Comparisons}
In this section, we consider $\bbeta^{\top} = (3, 1.5, 0, 0, 2, 0, 0, 0)$. Also, we simulated data which contains a training dataset, validation set and an independent test set. Note that the co-variates are scaled to have mean zero and unit variance. We fitted the models only using the training data and the tuning parameters were selected using the validation data. Finally, we computed the predictive mean absolute deviation (PMAD) criterion which is defined by
\begin{equation*}
\rm PMAD(\hbbeta_{\tau}^{\ast}) = \frac{1}{n_{test}}\sum_{i=1}^{n_{test}}\left | \by_{test}-\bX_{test}\hbbeta_{\tau}^{\ast} \right |.
\end{equation*}

\begin{table}[ht]
\centering
\begin{tabular}{rrcccc}
  \hline
&&\multicolumn{2}{c}{Case 1} 
&\multicolumn{2}{c}{Case 2} \\
  \cmidrule(lr){3-4} \cmidrule(lr){5-6}
$\tau$ & & \%10 & \%25 & \%10  & \%25\\ 
  \hline
0.25&FM & 0.335(0.012) & 0.302(0.010) & 1.668(0.033) & 1.380(0.029) \\ 
  &SM & 0.106(0.004) & 0.091(0.004) & 0.517(0.017) & 0.426(0.015) \\ 
  &PT & 0.108(0.006) & 0.094(0.006) & 0.625(0.046) & 0.515(0.038) \\ 
  &PS & 0.126(0.008) & 0.113(0.007) & 1.169(0.036) & 0.946(0.032) \\
  \cmidrule(lr){3-6}
  &Ridge & 0.247(0.007) & 0.217(0.006) & 0.716(0.008) & 0.662(0.009) \\ 
  &Lasso & 0.146(0.005) & 0.127(0.005) & 0.541(0.010) & 0.482(0.011) \\ 
  &ENET & 0.141(0.005) & 0.122(0.005) & 0.522(0.011) & 0.460(0.011) \\ 
   \hline
 0.5&  FM & 0.237(0.006) & 0.205(0.005) & 1.581(0.029) & 1.298(0.026) \\ 
  &SM & 0.071(0.003) & 0.060(0.002) & 0.479(0.016) & 0.397(0.013) \\ 
  &PT & 0.071(0.004) & 0.062(0.003) & 0.526(0.031) & 0.417(0.027) \\ 
  &PS & 0.103(0.005) & 0.091(0.004) & 1.015(0.032) & 0.788(0.027) \\ 
  \cmidrule(lr){3-6}
  &Ridge & 0.183(0.004) & 0.156(0.004) & 0.682(0.008) & 0.634(0.008) \\ 
  &Lasso & 0.096(0.004) & 0.088(0.003) & 0.521(0.010) & 0.455(0.010) \\ 
  &ENET & 0.093(0.003) & 0.086(0.003) & 0.483(0.011) & 0.425(0.010) \\ 
  \hline
0.75&  FM & 0.374(0.012) & 0.313(0.010) & 1.566(0.032) & 1.419(0.031) \\ 
  &SM & 0.104(0.005) & 0.084(0.004) & 0.486(0.016) & 0.424(0.015) \\ 
  &PT & 0.106(0.007) & 0.087(0.006) & 0.563(0.038) & 0.487(0.037) \\ 
  &PS & 0.134(0.008) & 0.101(0.007) & 1.076(0.034) & 0.889(0.033) \\ 
  \cmidrule(lr){3-6}
  &Ridge & 0.266(0.007) & 0.226(0.006) & 0.683(0.008) & 0.640(0.008) \\ 
  &Lasso & 0.151(0.006) & 0.130(0.005) & 0.521(0.011) & 0.487(0.011) \\ 
  &ENET & 0.145(0.006) & 0.123(0.005) & 0.489(0.011) & 0.450(0.010) \\ 
  \hline
Mean  &LSE& 1.180(0.075) & 0.929(0.059) & 1.231(0.078) & 1.002(0.063) \\ 
  \hline
\end{tabular}
\caption{PMAD values and their standard errors of listed estimators for  Cases 1 and 2
\label{tab:sim}}
\end{table}

We also use the notation $\cdot/\cdot/\cdot$ to describe the number of observations in the training, validation and test set respectively. Hence, we consider the each data set consists of $50/50/200$ observations and $\bX \sim N(\mathbf{0},\boldsymbol{\Sigma})$, where $\Sigma_{ij}=0.5^{|i-j|}$. Furthermore, the errors follow the one of the following distributions

\begin{itemize}
\item[Case 1:] 
$(1-\gamma)\mathcal{N}(0,1) + \gamma\left[ \frac{1}{\pi}\arctan(t)+\frac{1}{2}\right]$,
where the expression in square brackets denote the standard Cauchy distribution. The proportion $\gamma$ is often useful to verify the effect of outliers and small values of $\gamma$ lead to a contaminated normal distribution.  For example, $\gamma=0.1$ indicates 10\% outliers.

\item[Case 2:] We consider, 
$(1-\gamma)\mathcal{N}(0,1) + \gamma\mathcal{N}(0,100)$.
\end{itemize}


Table~\ref{tab:sim} represents PMAD values with standard errors in parenthesis $\gamma=0.1, 0.25$ indicate 10\% and 25\% outliers for both cases. According to these results, the PMAD of the SM estimation is the lowest since the null hypothesis is true. The PMAD of LSE is worse than the exiting methods since the errors are generated from contaminated distributions. On the other hand, the suggest methods perform better than penalty estimations in Case 1 while their performance is relative worse in Case 2. Regardless of the Cases, the proposed methods perform better than the full model estimation.


\section{Real Data Application}\label{sec5}

We implement the proposed strategies to the Hitters data which can be obtained from {\it ISLR} package of \cite{james2017package}. This data has $322$ observations of major league players on $20$ variables. 

We also omit missing values before we start to analyze. Hence, we have 263 observations. Furthermore, we apply Breusch-Pagan test ({\it bptest}) function in the {\it lmtest} package in R confirm that this data set has the problem of heteroskedasticity. In order to apply suggested methods, we first select the candidate sub model via BIC which confirms that the Hits, Walks and Years are significant covariates. Note that, one may use another sub-model selection criteria or model selection methods. Hence, we have two model which are the full model with all the predictors and the sub-model with predictors obtained by BIC. Finally, we may construct the pretest and shrinkage estimation techniques by combining the full-model and the sub-model in an optimal way.

\begin{table}[ht]
\caption{APE values for Hitters data}
\label{real:dat:results}
\centering
\begin{tabular}{rrrrr}
  \hline
 $\tau$& 0.25 & 0.5 & 0.75 \\ 
  \hline
  FM    & 4.256 & 4.232 & 4.520 \\ 
  SM    & 4.059 & 3.885 & 4.133 \\ 
  PT   & 4.180 & 4.089 & 4.486 \\ 
  PS    & 4.178 & 4.083 & 4.430 \\ 
  \cmidrule(lr){2-4}
  Ridge & 4.269 & 4.192 & 4.503 \\ 
  Lasso & 4.255 & 4.194 & 4.501 \\ 
  ENET  & 4.263 & 4.193 & 4.505 \\ 
  \hline
  & Mean \\
  \hline
  LSE & 4.587\\
   \hline
\end{tabular}
\end{table}

In the following, we divided the data into two parts randomly. One is the train, and the other is the test data. We fitted the model based on the train data. After that, we calculated the prediction errors by taking the mean absolute deviation of the observed and predicted values in the test set. In order to avoid random variation, this process is reiterated 999 times and is estimated the average prediction error (APE) that is given by 
\begin{equation*}
{\rm APE}(\hbbeta_{\tau}^{\star}) = \frac{1}{999} \sum_{k=1}^{999} \left(\frac{1}{\rm n_{test}} \sum_{i=1}^{\rm n_{test}}\left|
\by_{\rm test}-\bX_{\rm test}\hbbeta_{\tau}^{\star}
\right|\right),
\end{equation*} 
where $i$ and $k$ indicate observation and iteration, respectively.

Table \ref{real:dat:results} shows the results of Hitters data application. According to this results, the sub-model estimator has the best performance for each $\tau$ values since the candidate sub-model is selected truly. Also, the full model quantile regression estimation outshines LSE estimator, especially when $\tau=0.5$. This confirms that the Hitters data has not valid the assumptions of LSE. All suggested methods perform better than the full model estimator and penalty type quantile estimators. Furthermore, it can be concluded that the PT is less efficiency than PS.

\section{Conclusions}\label{sec6}
In this paper, we proposed preliminary test and shrinkage estimation strategies for linear quantile regression models. We established the theoretical properties of suggested estimators. Also, we conducted some Monte Carlo simulation studies and a real data application in order to investigate and compare the performance of listed estimators with some quantile type penalty estimators, namely Ridge, Lasso and Elastic Net, and LSE. According to the numerical studies, $\SM$ has the best performance if a candidate sub-model is selected true. As summary, the suggested methods perform better than LSE and penalty estimators. These results also consistent with our theory.

\section*{Appendix}
\begin{lemma} \label{lem_JB}
Let $\bX$ be $q-$dimensional normal vector distributed as $%
N\left( \boldsymbol{\mu }_{x},\boldsymbol{\Sigma }%
_{q}\right) ,$ then, for a measurable function of of $\varphi ,$ we have
\begin{align*}
\mathbb{E}\left[ \bX\varphi \left( \bX^{\top}\bX%
\right) \right] =&\boldsymbol{\mu }_{x}\mathbb{E}\left[ \varphi \chi _{q+2}^{2}\left(
\Delta \right) \right] \\
\mathbb{E}\left[ \boldsymbol{XX}^{\top}\varphi \left( \bX^{\top}%
\bX\right) \right] =&\boldsymbol{\Sigma }_{q}\mathbb{E}\left[ \varphi
\chi _{q+2}^{2}\left( \Delta \right) \right] +\boldsymbol{\mu }_{x}\boldsymbol{\mu }_{x}^{\top}\mathbb{E}\left[ \varphi \chi
_{q+4}^{2}\left( \Delta \right) \right]
\end{align*}
where $\chi_{v}^{2}\left( \Delta \right)$ is a non-central chi-square distribution with $v$ degrees of freedom and non-centrality parameter $\Delta$.
\end{lemma}
\begin{proof} It can be found in \cite{judge1978bock} \end{proof}
\begin{proof}[Proof of Proposition \ref{prop_vector_dist}]
Using the definition of asymptotic bias and $\SM=\FM+\bGamma_{11}^{-1}\bGamma_{12}\hbbeta_{2,\tau}^{\rm FM}$, we have
\begin{eqnarray*}
\mathcal{B}\left(\FM \right)&=&\mathbb{E}\left\{\underset{%
n\rightarrow \infty }{\lim }\sqrt{n}\left( \FM -\btau \right)\right\}
=\mathbf{0}_{p_1} \\
\mathcal{B}\left( \SM\right)&=&\mathbb{E}\left\{ \underset{n\rightarrow \infty }{\lim }\sqrt{n}\left( \SM -\btau\right) \right\} 
=\mathbb{E}\left\{ \underset{n\rightarrow \infty }{\lim }\sqrt{n}\left( \FM +\bGamma_{11}^{-1}\bGamma_{12}\boldsymbol{%
\widehat{\beta}}_{2,\tau}^{\rm FM}-\btau\right) \right\} 
=\bGamma_{11}^{-1}\bGamma_{12}\bgamma 
=\bdelta
\end{eqnarray*}
and also using the definition of conditional expectation, one may directly get
\begin{eqnarray*}
\bvt _{1}  &\sim& \mathcal{N} \left(\boldsymbol{0 }_{p_1}, \tau(1-\tau)\bGamma_{11.2}^{-1} \right),\\
\bvt _{2}  &\sim& \mathcal{N} \left(\bdelta, \tau(1-\tau)\bGamma_{11}^{-1} \right).
\end{eqnarray*}
We also compute the asymptotic covariance matrix of $\bvt_3$ as follows:
\begin{eqnarray*}
\Ups\left(\bvt_3,\bvt_3'\right)&=&Cov\left(\FM-\SM,\FM-\SM\right) 
=\bGamma_{11}^{-1}\bGamma_{12} {\mathcal Var}\left( \hbbeta_{2,\tau}^{\rm FM}\right) \bGamma_{21}\bGamma_{11}^{-1}\\
&=&\tau(1-\tau)\bGamma_{11}^{-1}\bGamma_{12}\bGamma_{22.1}^{-1} \bGamma_{21}\bGamma_{11}^{-1}
=\bPhi.
\end{eqnarray*}
Thus, $\bvt_3\sim{\mathcal{N}\left(-\bdelta,\bPhi\right)}$. We also need to compute $Cov\left(\bvt_1,\bvt_3\right)$ and $Cov\left(\bvt_2,\bvt_3\right)$. First we compute
\begin{eqnarray*}
Cov\left(\FM,\SM\right)&=&Cov\left(\FM,\FM +\bGamma_{11}^{-1}\bGamma_{12}\hbbeta_{2,\tau}^{\rm FM}\right) 
=Cov\left(\FM,\FM \right)+Cov\left(\FM,\bGamma_{11}^{-1}\bGamma_{12}\hbbeta_{2,\tau}^{\rm FM}\right)\\
&=&\tau(1-\tau)\bGamma_{11.2}^{-1}+\tau(1-\tau)\bGamma_{12}\bGamma_{21}\bGamma_{11}^{-1},
\end{eqnarray*}
then
\begin{eqnarray*}
Cov\left(\bvt_1,\bvt_3\right)&=&Cov\left(\FM,\hbbeta_{1,\tau}^{\rm FM}-\SM\right) 
=Cov\left(\FM,\FM\right)-Cov\left(\FM,\SM\right)\\
&=&-\tau(1-\tau)\bGamma_{12}\bGamma_{21}\bGamma_{11}^{-1}
=\bSigma_{12},\\
Cov\left(\bvt_2,\bvt_3\right)&=&Cov\left(\SM,\FM-\SM\right)
=Cov\left(\SM,\FM\right)-Cov\left(\SM,\SM\right) \\
&=& \tau(1-\tau)\left(\bGamma_{11.2}^{-1}+\bGamma_{12}\bGamma_{21}\bGamma_{11}^{-1} - \bGamma_{11}\right)
= \bSigma^*
\end{eqnarray*}
\end{proof}
\begin{proof}[Proof of Theorem \ref{ADB-Low-Linear}] The expressions of $\mathcal{B}\left(\FM\right)=\boldsymbol 0$ and $\mathcal{B}\left( \SM\right)=\bdelta$ are directly obtained from the Proposition \ref{prop_vector_dist}. The rests are also given as follows:

\begin{eqnarray*}
\mathcal{B}\left( \PT \right)&=&\mathbb{E}\left\{ \underset{n\rightarrow \infty }{\lim }\sqrt{n}\left( \PT -\btau\right) \right\} \\
&=&\mathbb{E}\left\{ \underset{n\rightarrow \infty }{\lim }\sqrt{n}\left( \FM -\btau\right) \right\}-\mathbb{E}\left\{ \underset{n\rightarrow \infty }{\lim }\sqrt{n}\left( \FM -\SM \right)\textrm{I}\left(\mathcal{W}_{n}\leq  \chi^2_{p_2,\alpha}\right) \right\}\\
&=&\bdelta H_{p_{2}+2}\left( \chi _{p_{2},\alpha
}^{2};\Delta \right)\\
\mathcal{B}\left( \SS \right)&=&\mathbb{E}\left\{ \underset{n\rightarrow \infty }{\lim }\sqrt{n}\left( \SS -\btau \right) \right\} \\
&=&\mathbb{E}\left\{ \underset{n\rightarrow \infty }{\lim }\sqrt{n}\left( \FM -\btau \right) \right\}-\mathbb{E}\left\{ \underset{n\rightarrow \infty }{\lim }\sqrt{n}\left( \FM -\SM \right)(p_{2}-2)\mathcal{W}_{n}^{-1} \right\}\\
&=&(p_{2}-2)\bdelta\mathbb{E}\left\{\chi _{p_{2}+2}^{-2}\left(\Delta\right)\right\}\\
\mathcal{B}\left( \PS \right)&=&\mathbb{E}\left\{ \underset{n\rightarrow \infty }{\lim }\sqrt{n}\left( \PS -\btau \right) \right\} \\
&=&\mathbb{E}\left\{ \underset{n\rightarrow \infty }{\lim }\sqrt{n}\left[ \SM +\left(\FM-\SM\right)\left(1-\textrm{I}(\mathcal{W}_{n}\leq p_2-2)\right)\right. \right. \\ 
&&\left.  \left.-\left(\FM-\SM\right)\left(p_2-2\right)\mathcal{W}_{n}^{-1}\textrm{I}\left(\mathcal{W}_{n}>p_{2}-2\right)-\btau \right] \right\}\\
&=&\bdelta H_{p_{2}+2}\left( p_2 -2;\Delta \right) +(p_{2}-2)\boldsymbol{\delta }\mathbb{E}\left\{ \chi _{p_{2}+2 }^{-2}\left(
\Delta \right)\textrm{I}\left( \chi _{p_{2}+2 }^{2}\left( \Delta \right)
>p_{2}-2\right) \right\}
\end{eqnarray*}
\end{proof}
\begin{proof}[Proof of Theorem~\ref{risk}]
The asymptotic covariance of suggested estimators are obtained as follows:

The asymptotic covariance of $\PT$ is
given by
\begin{eqnarray*}
\Ups\left( \FM \right) &=&\mathbb{E}\left\{
\underset{n\rightarrow \infty }{\lim }{n}\left( \FM -\btau \right)\left( \FM -\btau \right) ^{'}\right\}
= Cov\left( \bvt _{1},\bvt _{1}^{'}\right) +\mathbb{E}\left( \bvt_{1}\right) \mathbb{E}\left( \bvt_{1}^{'}\right) 
=\tau(1-\tau)\bGamma_{11.2}^{-1}\\
\Ups\left( \SM \right) &=&\mathbb{E}\left\{
\underset{n\rightarrow \infty }{\lim }{n}\left( \SM -\btau \right)\left( \SM -\btau \right) ^{'}\right\}
= Cov\left( \bvt _{2},\bvt _{2}^{'}\right) +\mathbb{E}\left( \bvt_{2}\right) \mathbb{E}\left( \bvt_{2}^{'}\right) 
=\tau(1-\tau)\bGamma_{11}^{-1}+\bdelta \bdelta^{'} \\
\Ups\left( \PT \right)
&=&\mathbb{E}\left\{ \underset{n\rightarrow \infty }{\lim }{n}\left( \PT -\btau\right)\left( \PT -\btau\right) ^{'}\right\} \no \\
&=&\mathbb{E}\left\{ \underset{n\rightarrow \infty }{\lim }n\left[ \left( \FM -\btau \right) -\left( \FM -\SM \right) \textrm{I}\left( %
\mathcal{W}_{n}< \chi^2_{p_2,\alpha} \right) \right] \right.\\
&&\times \left. \left[ \left( \FM -\btau\right) -\left( \FM -\SM \right) \textrm{I}\left( \mathcal{W}_{n}< \chi^2_{p_2,\alpha} \right) \right] ^{'}\right\} \no \\
&=&\mathbb{E}\left\{ \left[ \bvt_{1}-\bvt_{3}\textrm{I}\left( \mathcal{W}_{n}<
c_{n,\alpha }\right) \right] \left[ \bvt_{1}-\bvt_{3}\textrm{I} \left(\mathcal{W}_{n}< \chi^2_{p_2,\alpha} \right) \right] ^{'}\right\} \\
&=&\mathbb{E}\left\{ \bvt_{1}\bvt_{1}^{'}-2\bvt_{3}\bvt_{1}^{'}\textrm{I}\left( \mathcal{W}_{n}< \chi^2_{p_2,\alpha}\right) +\bvt_{3}\bvt_{3}^{'}\textrm{I}\left( \mathcal{W}_{n}< \chi^2_{p_2,\alpha} \right)
\right\} \text{.}
\end{eqnarray*}
Considering, $\mathbb{E}\left\{ \bvt_{1}\bvt_{1}' \right\} = \tau(1-\tau)\bGamma_{11.2}^{-1}$ 
and 
$
\mathbb{E}\left\{ \bvt_{3}\bvt_{3}^{'}I\left( \mathcal{W}_{n}<
\chi _{p_{2},\alpha}^2
\right) \right\} =\bPhi \mathbb{H}_{p_{2}+2}\left(
\chi _{p_{2},\alpha }^{2};\Delta \right) +\bdelta \bdelta ^{'}\mathbb{H}_{p_{2}+4}\left( \chi _{p_{2},\alpha
}^{2};\Delta \right)$ by \cite{judge1978bock}, we have the following
\begin{eqnarray*}
\mathbb{E}\left\{ \bvt_{3}\bvt_{1}'I\left( \mathcal{W}_{n}< \chi^2_{p_2,\alpha} \right) \right\} &=& \mathbb{E}\left\{ \mathbb{E}\left( \bvt_{3}\bvt_{1}^{'}\textrm{I}\left( \mathcal{W}%
_{n} < \chi^2_{p_2,\alpha} \right) |\bvt_{3}\right) \right\} =\mathbb{E}\left\{ \bvt_{3}\mathbb{E}\left( \bvt_{1}^{'}\textrm{I}\left( \mathcal{W}%
_{n}\leq \chi^2_{p_2,\alpha} \right) |\bvt_{3}\right) \right\} \\
&=&\mathbb{E}\left\{ \bvt_{3}\left( \boldsymbol 0+\bSigma_{12}\bPhi^{-1}\left(\bvt_3+\bdelta\right)\right)'\textrm{I}\left( \mathcal{W}%
_{n} < \chi^2_{p_2,\alpha} \right) \right\} \\
&=& \mathbb{E}\left\{ \bvt_{3}\bvt_{3}^{'}\bPhi^{-1}\bSigma_{21}\textrm{I}\left(\mathcal{W}%
_{n}< \chi^2_{p_2,\alpha} \right) \right\}+\mathbb{E}\left\{ \bvt _{3}\bdelta^{'}\bPhi^{-1}\bSigma_{21}\textrm{I}\left(\mathcal{W}%
_{n}< \chi^2_{p_2,\alpha} \right) \right\} \\
&=&\left[ \bPhi \mathbb{H}_{p_{2}+2}\left(
\chi _{p_{2},\alpha }^{2};\Delta \right) +\bdelta \bdelta ^{'}\mathbb{H}_{p_{2}+4}\left( \chi _{p_{2},\alpha
}^{2};\Delta \right)\right]\bPhi^{-1}\bSigma_{21} +\bdelta\bdelta^{'}\bPhi^{-1}\bSigma_{21}\mathbb{H}_{p_{2}+2}\left( \chi_{p_{2},\alpha}^{2};\Delta \right)\\
&=& \bSigma_{21} \mathbb{H}_{p_{2}+2}\left(
\chi _{p_{2},\alpha }^{2};\Delta \right)+ \bdelta \bdelta ^{'}\bPhi^{-1}\bSigma_{21}\left[\mathbb{H}_{p_{2}+4}\left( \chi _{p_{2},\alpha
}^{2};\Delta \right)+\mathbb{H}_{p_{2}+2}\left( \chi _{p_{2},\alpha
}^{2};\Delta \right)\right]
\end{eqnarray*}
So finally we have,
\begin{eqnarray}\label{cov_PT}
\Ups\left( \PT \right)
&=&\tau(1-\tau)\bGamma_{11.2}^{-1}-2\bSigma_{21} \mathbb{H}_{p_{2}+2}\left(
\chi _{p_{2},\alpha }^{2};\Delta \right)+\bdelta \bdelta ^{'}\bPhi^{-1}\bSigma_{21}\left[\mathbb{H}_{p_{2}+2}\left( \chi_{p_{2}}^{2};\Delta \right)-2 \mathbb{H}_{p_{2}+4}\left( \chi _{p_{2},\alpha
}^{2};\Delta \right)\right] \no
\\&&+\bPhi \mathbb{H}_{p_{2}+2}\left(
\chi _{p_{2},\alpha }^{2};\Delta \right) +\bdelta \bdelta ^{'}\mathbb{H}_{p_{2}+4}\left( \chi _{p_{2},\alpha
}^{2};\Delta \right)
\end{eqnarray}
The asymptotic covariance of $\SS$ is
given by
\begin{eqnarray*}
\Ups\left( \SS \right)
&=&\mathbb{E}\left\{ \underset{n\rightarrow \infty }{\lim }{n}\left( \SS-\btau \right)\left(
\SS-\btau \right) '\right\}  \\
&=&\mathbb{E}\left\{ \underset{n\rightarrow \infty }{\lim }n\left[ \left( \FM -\btau \right) -\left(
\FM-\SM\right) (p_{2}-2) \mathcal{W}_{n}^{-1}\right] \right. \\
&&\times \left. \left[ \left( \FM-\btau \right) -\left( \FM-\SM\right) (p_{2}-2) \mathcal{W}_{n}^{-1}%
\right] ^{'}\right\}  \\
&=&\mathbb{E}\left\{ \bvt_{1}\bvt_{1}'-2(p_{2}-2)
\bvt_{3}\bvt_{1}^{'}\mathcal{W}_{n}^{-1}+(p_{2}-2)^{2}\bvt _{3}\bvt_{3}^{'}\mathcal{W}%
_{n}^{-2}\right\} \text{.}
\end{eqnarray*}%
Now, by using Lemma~\ref{lem_JB}, we have
$
\mathbb{E}\left\{ \bvt_{3}\bvt_{3}^{'}\mathcal{W}_{n}^{-2}\right\}=\bPhi \mathbb{E}\left(\chi _{p_{2}+2}^{-4}\left(\Delta\right) \right) +\bdelta \bdelta '\mathbb{E}\left(\chi _{p_{2}+4}^{-4}\left(\Delta\right) \right)
$
 and also,
\begin{eqnarray*}
\mathbb{E}\left\{ \bvt_{3}\bvt_{1}'\mathcal{W}_{n}^{-1}\right\}
&=&\mathbb{E}\left\{ \mathbb{E}\left( \bvt_{3}\bvt_{1}'\mathcal{W}%
_{n}^{-1}|\bvt_{3}\right) \right\} 
=\mathbb{E}\left\{ \bvt_{3}\mathbb{E}\left( \bvt_{1}'\mathcal{W}_{n}^{-1}|\bvt_{3}\right) \right\}
=\mathbb{E}\left\{ \bvt _{3}\left(\bSigma_{12}\bPhi^{-1}\left(\bvt_3+\bdelta\right)\right)^{'}\mathcal{W}%
_{n}^{-1} \right\} \\
&=&\mathbb{E}\left\{\bvt_3\bvt_3'\bPhi^{-1}\bSigma_{21}\mathcal{W}_{n}^{-1} \right\} +\mathbb{E}\left\{\bvt_3\bdelta^{'}\bPhi^{-1}\bSigma_{21}\mathcal{W}_{n}^{-1}\right\} \\
&=&\left[ \bPhi \mathbb{E}\left\{\chi _{p_{2}+2}^{-2}\left(\Delta\right)\right\} +\bdelta \bdelta ^{'}\mathbb{E}\left\{\chi _{p_{2}+4}^{-2}\left(\Delta\right)\right\}\right]\bPhi^{-1}\bSigma_{21}+\bdelta\bdelta^{'}\bPhi^{-1}\bSigma_{21}\mathbb{E}\left\{\chi _{p_{2}+2}^{-2}\left(\Delta\right)\right\}\\
&=& \bSigma_{21}\mathbb{E}\left\{\chi _{p_{2}+2}^{-2}\left(\Delta\right)\right\}+ \bdelta \bdelta'\bPhi^{-1}\bSigma_{21} \left[ \mathbb{E}\left\{\chi _{p_{2}+4}^{-2}\left(\Delta\right)\right\}+\mathbb{E}\left\{\chi _{p_{2}+2}^{-2}\left(\Delta\right)\right\}\right]
\end{eqnarray*}%
Therefore, we obtain $\Ups\left( \SS \right)$ by combining all of the components:
\begin{eqnarray}\label{cov_S}
\Ups\left( \SS \right)&=&\tau(1-\tau)\bGamma_{11.2}^{-1}-2(p_{2}-2)\left\{\bSigma_{21}\mathbb{E}\left\{\chi _{p_{2}+2}^{-2}\left(\Delta\right)\right\}+ \bdelta \bdelta'\bPhi^{-1}\bSigma_{21} \left[ \mathbb{E}\left\{\chi _{p_{2}+4}^{-2}\left(\Delta\right)\right\}+\mathbb{E}\left\{\chi _{p_{2}+2}^{-2}\left(\Delta\right)\right\}\right]\right\} \no\\
&&+(p_{2}-2)^2\left\{\bPhi \mathbb{E}\left(\chi _{p_{2}+2}^{-4}\left(\Delta\right) \right) +\bdelta \bdelta '\mathbb{E}\left(\chi _{p_{2}+4}^{-4}\left(\Delta\right) \right)\right\}.
\end{eqnarray}
Finally, we compute  $\Ups\left( \PS \right)$ using
$\PS = \SS-\left(\FM-\SM \right)\left(1-(p_{2}-2)\mathcal{W}_{n}^{-1}\right)\textrm{I}\left(\mathcal{W}_n\leq p_{2}-2\right)$
as follows
\begin{eqnarray*}
\Ups\left( \PS \right)&=&\mathbb{E}\left\{\underset{n\rightarrow \infty }{\lim }{n}\left( \PS -\btau \right)\left( \PS -\btau \right) ^{'}\right\} \\
&=&\mathbb{E}\left\{\underset{n\rightarrow \infty }{\lim }{n}\left[\left(\SS  -\btau \right)-\left(\FM -\SM \right)\left(1-(p_{2}-2)\mathcal{W}_n^{-1}\right)\textrm{I}\left(\mathcal{W}_n\leq (p_{2}-2)\right)\right] \right.
\\
&&\times \left.\left[\left(\SS -\btau \right)-\left(\FM-\SM \right)\left(1-(p_{2}-2)\mathcal{W}_n^{-1}\right)\textrm{I}\left(\mathcal{W}_n \leq (p_{2}-2)\right)\right]^{'}
\right\} \\
&=&\Ups\left( \SS \right)-2\mathbb{E}\left\{\underset{n\rightarrow \infty }{\lim }{n}\left(\FM-\SM \right)\left(\SS -\btau \right)\left(1-(p_{2}-2)\mathcal{W}_n^{-1}\right)\textrm{I}\left(\mathcal{W}_n \leq (p_{2}-2)\right)\right\} \\
&&+\mathbb{E}\left\{\underset{n\rightarrow \infty }{\lim }{n}\left(\FM -\SM \right)\left(\FM -\SM \right)^{'}\left(1-(p_{2}-2)\mathcal{W}_n^{-1}\right)^{2}\textrm{I}\left(\mathcal{W}_n \leq (p_{2}-2)\right)\right\}\\
&=&\Ups\left( \SS \right)-2\mathbb{E}\left\{\underset{n\rightarrow \infty }{\lim }{n}\left(\FM-\SM \right)\left[\left(\FM -\btau \right)-(p_{2}-2)\left(\FM-\SM \right)\mathcal{W}_{n}^{-1}\right]' \right. \\
&&\times\left.\left(1-(p_{2}-2)\mathcal{W}_n^{-1}\right)\textrm{I}\left(\mathcal{W}_n\right)\right\} \\
&&+\mathbb{E}\left\{\underset{n\rightarrow \infty }{\lim }{n}\left(\FM-\SM\right)\left(\FM -\SM \right)^{'}\left(1-(p_{2}-2)\mathcal{W}_n^{-1}\right)^{2}\textrm{I}\left(\mathcal{W}_n\leq p_{2}-2\right)\right\} \\
&=&\Ups\left( \SS \right)-2\mathbb{E}\left\{\bvt_3\bvt_1'\left(1-(p_{2}-2)\mathcal{W}_n^{-1}\right)\textrm{I}\left(\mathcal{W}_n\leq (p_{2}-2)\right)\right\}\\
&&-2(p_{2}-2)\mathbb{E}\left\{\bvt_3\bvt_3^{'}\mathcal{W}_n^{-1}\left(1-(p_{2}-2)\mathcal{W}_n^{-1}\right)\textrm{I}\left(\mathcal{W}_n\leq p_{2}-2\right)\right\}\\
&&+\mathbb{E}\left\{\bvt_3 \bvt_3'\left(1-(p_{2}-2)\mathcal{W}_n^{-1}\right)^{2}\textrm{I}\left(\mathcal{W}_n\leq p_{2}-2\right)\right\} \\
&=&\Ups\left( \SS\right)-2\mathbb{E}\left\{\bvt_3\bvt_1^{'}\left(1-(p_{2}-2)\mathcal{W}_n^{-1}\right)\textrm{I}\left(\mathcal{W}_n\leq p_{2}-2\right)\right\} \\
&&-\mathbb{E}\left\{\bvt_3\bvt_3^{'}(p_{2}-2)^2\mathcal{W}_n^{-2}\textrm{I}\left(\mathcal{W}_n\leq p_{2}-2\right)\right\}+\mathbb{E}\left\{\bvt_3\bvt_3^{'}\textrm{I}\left(\mathcal{W}_n\leq p_{2}-2\right)\right\}.
\end{eqnarray*}
So, we need the following identities:
\begin{eqnarray*}
\mathbb{E}\left\{\bvt_3\bvt_3^{'}\textrm{I}\left(\mathcal{W}_n\leq p_{2}-2\right)\right\}&=&\bPhi\mathbb{H}_{p_{2}+2}\left(
p_{2}-2;\Delta \right)+\bdelta\bdelta^{'}\mathbb{H}_{p_{2}+4}\left(
p_{2}-2;\Delta \right), \\
\mathbb{E}\left\{\bvt_3\bvt_3^{'}\mathcal{W}_n^{-2}\textrm{I}\left(\mathcal{W}_n\leq p_{2}-2\right)\right\} &=& \bPhi \mathbb{E}\left\{\chi _{p_{2}+2}^{-4}\left(\Delta \right)\textrm{I}\left(\chi _{p_{2}+2}^{2}\left(\Delta\right)\leq p_{2}-2 \right)\right\}\\
&&+\bdelta\bdelta^{'}\mathbb{E}\left\{\chi _{p_{2}+2}^{-4}\left(\Delta \right)\textrm{I}\left(\chi _{p_{2}+2}^{2}\left(\Delta\right)\leq p_{2}-2 \right)\right\},
\end{eqnarray*}
and
\begin{eqnarray*}
\lefteqn{\mathbb{E}\left\{\bvt_3\bvt_1^{'}\left(1-(p_{2}-2)\mathcal{W}_n^{-1}\right)\textrm{I}\left(\bW_n\leq p_{2}-2\right)\right\} }\\
&=&\mathbb{E}\left\{\mathbb{E}\left[\bvt_3\bvt_1^{'}\left(1-(p_{2}-2)\mathcal{W}_n^{-1}\right)\textrm{I}\left(\bW_n\leq p_{2}-2\right)|\bvt_3\right]\right\} \\
&=&\mathbb{E}\left\{\bvt_3\mathbb{E}\left[\bvt_1^{'}\left(1-(p_{2}-2)\mathcal{W}_n^{-1}\right)\textrm{I}\left(\bW_n\leq p_{2}-2\right)|\bvt_3\right]\right\} \\
&=&\mathbb{E}\left\{\bvt_3\bvt_3^{'}\bPhi^{-1}\bSigma_{21}\left(1-(p_{2}-2)\mathcal{W}_n^{-1}\right)\textrm{I}\left(\bW_n\leq p_{2}-2\right)\right\} \\
&&+\mathbb{E}\left\{\bvt_3\bdelta^{'}\bPhi^{-1}\bSigma_{21}\left(1-(p_{2}-2)\mathcal{W}_n^{-1}\right)\textrm{I}\left(\bW_n\leq p_{2}-2\right)\right\} \\
&=&\bSigma_{21}\mathbb{E}\left\{\left(1-(p_{2}-2)\chi _{p_{2}+2}^{-2}\left(\Delta\right)\right)\textrm{I}\left(\chi _{p_{2}+4}^{2}\left(\Delta\right)\leq p_{2}-2\right)\right\}\\
&&+\bdelta\bdelta^{'}\bPhi^{-1}\bSigma_{21}\mathbb{E}\left\{1-(p_{2}-2)\chi _{p_{2}+4}^{-2}\left(\Delta\right)\textrm{I}\left(\chi _{p_{2}+4}^{2}\left(\Delta\right)\leq p_{2}-2\right)\right\} \\
&&+\bdelta\bdelta^{'}\bPhi^{-1}\bSigma_{21}\mathbb{E}\left\{1-(p_{2}-2)\chi _{p_{2}+2}^{-2}\left(\Delta\right)\textrm{I}\left(\chi _{p_{2}+2}^{2}\left(\Delta\right)\leq p_{2}-2\right)\right\}. 
\end{eqnarray*}
Therefore, we obtain
\begin{eqnarray}\label{cov_PS}
\Ups\left( \PS \right)&=&\Ups\left( \SS\right)-2\bSigma_{21}\mathbb{E}\left\{\left(1-(p_{2}-2)\chi _{p_{2}+2}^{-2}\left(\Delta\right)\right)\textrm{I}\left(\chi _{p_{2}+4}^{2}\left(\Delta\right)\leq p_{2}-2\right)\right\}\no\\
&&-2\bdelta\bdelta^{'}\bPhi^{-1}\bSigma_{21}\mathbb{E}\left\{1-(p_{2}-2)\chi _{p_{2}+4}^{-2}\left(\Delta\right)\textrm{I}\left(\chi _{p_{2}+4}^{2}\left(\Delta\right)\leq p_{2}-2\right)\right\}\no \\
&&-2\bdelta\bdelta^{'}\bPhi^{-1}\bSigma_{21}\mathbb{E}\left\{1-(p_{2}-2)\chi _{p_{2}+2}^{-2}\left(\Delta\right)\textrm{I}\left(\chi _{p_{2}+2}^{2}\left(\Delta\right)\leq p_{2}-2\right)\right\}\no \\
&&-(p_{2}-2)^2\bPhi \mathbb{E}\left\{\chi _{p_{2}+2}^{-4}\left(\Delta \right)\textrm{I}\left(\chi _{p_{2}+2}^{2}\left(\Delta\right)\leq p_{2}-2 \right)\right\}\no\\
&&-(p_{2}-2)^2\bdelta\bdelta^{'}\mathbb{E}\left\{\chi _{p_{2}+2}^{-4}\left(\Delta \right)\textrm{I}\left(\chi _{p_{2}+2}^{2}\left(\Delta\right)\leq p_{2}-2 \right)\right\}\no\\
&&+\bPhi\mathbb{H}_{p_{2}+2}\left(
p_{2}-2;\Delta \right)+\bdelta\bdelta^{'}\mathbb{H}_{p_{2}+4}\left(
p_{2}-2;\Delta \right).
\end{eqnarray}
Now, one can obtain the risks of the listed estimators by using the equation ~\ref{risk_def}.
\end{proof}

\end{document}